\documentclass[12pt,reqno,twoside]{amsart}

\usepackage{amsthm,amsmath,amssymb,mathrsfs,amsbsy}
\usepackage{graphicx,epsfig,wrapfig,sidecap,subfig}
\usepackage[all]{xy}
\usepackage{proof}
\usepackage[usenames]{color}
\usepackage{hyperref}
\usepackage{bm}
\usepackage[capitalize]{cleveref}

\newcommand\name[1]{\label{#1}{\ifdraft{\sn [#1]}\else\ignorespaces\fi}}

\newcommand\eq[2]{{\ifdraft{\ \tt [#1]}\else\ignorespaces\fi}\begin{equation}\label{#1}{#2}\end{equation}}
\newcommand {\equ}[1]{\eqref{#1}}

\newcommand{\spa}{{\rm span}}

\newcommand{\diag}{{\rm diag}}
\newcommand{\diam}{{\rm diam}}

\font\sn = cmssi8 scaled \magstep0

\long\def\combarak#1{\ifdraft{\color{red}\sn Barak says ``#1'' }\else\ignorespaces\fi}

\xyoption{dvips}
\xyoption{color}

\newif\ifdraft\drafttrue
\draftfalse
\usepackage{color}

\title[Problems of Danzer and Gowers and dynamics]{On problems of Danzer and Gowers and dynamics on the space of closed subsets of $\R^d$}
\date{}
\author{Omri Solan}
\address{School of Mathematical Sciences, Tel Aviv University, Israel}
\email{omrisola@mail.tau.ac.il}
\author{Yaar Solomon}
\address{Department of Mathematics, Stony Brook University, Stony Brook, NY}
\email{yaar.solomon@stonybrook.edu}
\urladdr{\url{http://www.math.stonybrook.edu/~yaars/}}
\author{Barak Weiss}
\address{School of Mathematical Sciences, Tel Aviv University, Israel}
\email{barakw@post.tau.ac.il}
\urladdr{\url{http://www.math.tau.ac.il/~barakw/}}

\newcommand{\N}{{\mathbb{N}}}

\newcommand{\vre}{\varepsilon}
\newcommand{\Z}{{\mathbb{Z}}}

\newcommand{\R}{{\mathbb{R}}}

\newcommand{\ee}{{\mathbf{e}}}

\newcommand{\xx}{\mathbf{x}}
\newcommand{\yy}{\mathbf{y}}
\newcommand{\pp}{\mathbf{p}}
\newcommand{\X}{\mathcal{X}}
\newcommand{\Y}{\mathcal{Y}}

\newcommand{\df}{{\, \stackrel{\mathrm{def}}{=}\, }}

\newcommand{\absolute}[1] {\left|{#1}\right|}

\newcommand{\norm}[1]{\left\|{#1}\right\|}

\newcommand{\Vol}{\mathrm{Vol}}

\DeclareMathOperator{\SL}{SL}

\newcommand {\ignore}[1]  {}

\theoremstyle{plain}
\newtheorem{thm}{Theorem}[section]
\newtheorem{lem}[thm]{Lemma}
\newtheorem{prop}[thm]{Proposition}
\newtheorem{cor}[thm]{Corollary}

\newtheorem{ques}[thm]{Question}

\theoremstyle{definition}

\newtheorem{question}[thm]{Question}

\numberwithin{equation}{section}
\swapnumbers

\begin{document}

\begin{abstract}
Considering the space of closed subsets of $\R^d$, endowed with the
Chabauty-Fell topology, and the affine action of
$\SL_d(\R)\ltimes\R^d$, we prove that the only minimal subsystems are
the fixed points $\{\varnothing\}$ and $\{\R^d\}$. As a consequence we
resolve a question of Gowers concerning the existence of certain
Danzer sets: there is no set $Y \subset \R^d$ such that for every
convex set $\mathcal{C} \subset \R^d$ of volume one, the cardinality
of $\mathcal{C} \cap Y$ is bounded above and below by nonzero contants
independent of $\mathcal{C}$. We also provide a short independent
proof of this fact and deduce a quantitative consequence: for every
$\varepsilon$-net $N$ for convex sets in $[0,1]^d$ there is a convex
set of volume $\varepsilon$ containing at least
$\Omega(\log\log(1/\varepsilon))$ points of $N$.  
\end{abstract}

\maketitle

\section{Introduction}\label{sec:introduction}\label{sec:Intro}
A set $Y\subseteq\R^d$ is called a \emph{Danzer set} if there exists
an $s>0$ such that $Y$ intersects every convex set of volume $s$. The
following question is due to Danzer, and is open since the sixties,
see \cite{Fenchel}, \cite[p. 148]{CFG},  \cite[p. 288]{GL}: is there a
Danzer set $Y\subseteq\R^d\:\: (d\ge 2)$ such that 
with a bounded upper density,
i.e. with 
\[\limsup_{R\to\infty}\frac{\#(B(0,R)\cap Y)}{R^d} < \infty ?\]

Different authors have asked about variants of Danzer's problem. 
In \cite{Gowers} Gowers asked whether  there exists a set $Y\subseteq\R^d$, and $C>0$, such that for every convex set $K$ of volume $1$ we have 
\[1\le\#(K\cap Y)\le C.\]

In this paper we answer Gowers' question negatively, namely: 
\begin{thm}\label{thm:No_Gowers-Danzer_set}
Let $Y\subseteq\R^d$ be a Danzer set, then for every $n\in\N$ and $\varepsilon>0$ there is an ellipsoid $E_n$ with $\Vol(E_n)<\varepsilon$ and $\#(E_n\cap Y)\ge n$. 
\end{thm}

In fact we give two different proofs of Theorem
\ref{thm:No_Gowers-Danzer_set}. Our first proof is dynamical. We
denote by $\X$ the space of closed subsets of $\R^d$ equipped with the
Chabauty-Fell topology, and let $G\df\SL_d(\R)\ltimes\R^d$ denote the
group of affine transformations of $\R^d$ which preserve Lebesgue
measure and orientation. The action of this group on $\R^d$ induces a
natural action on $\X$. Denote by $$
U_0 \df \left\{ u(\mathbf{a}): \mathbf{a} \in
  \R^{d-1} \right \}, $$ 
where 
\begin{equation}\label{eq: def U}
u(\mathbf{a}) \df
\begin{pmatrix} 1&a_2&a_3&\cdots& a_{d}\\ 0&1&0&\cdots&0 \\
  \vdots& &\ddots& &\vdots \\ \vdots&& &\ddots&0 \\
  0&\cdots&\cdots&0&1 \end{pmatrix}, \ \ \text{ for } \mathbf{a}  = (a_2, \ldots, a_{d}) ,
\end{equation}
and let $U = U_0 \ltimes \R^d$. 
Also let 
\eq{eq: defn gt}{
g_t \df \diag \left (e^{(d-1)t}, e^{-t}, \ldots, e^{-t} \right),
}
let $H_0$
be the subgroup of $\SL_d(\R)$ generated by $\{g_t: t \in \R\}$ and $U_0$, and
let $H  \df H_0 \ltimes \R^d$.

Our main result about this dynamical system is:
\begin{thm}\label{thm:classifying_minimal_subsystems}
For every $F\in\X, $ either $\varnothing\in \overline{H.F}$ or $\R^d\in\overline{U.F}$.
\end{thm}
As we will show, 
Theorem \ref{thm:No_Gowers-Danzer_set} is a straightforward
consequence of Theorem \ref{thm:classifying_minimal_subsystems}.  
Theorem
\ref{thm:classifying_minimal_subsystems} also immediately yields a classification of the
minimal subsystems of the dynamical system $(\X,G)$. Recall that a
subset $\Y\subseteq \X$ is \emph{minimal} if it is non-empty, closed,
$G$-invariant, and minimal with respect to inclusion with these
properties. Fixed points are obvious examples of minimal subsystems. 
The following is an
immediate consequence of Theorem \ref{thm:classifying_minimal_subsystems}:

\begin{cor}\label{cor:classifying_minimal_subsystems}
The only minimal subsystems of $(\X,G)$ are the fixed points $\{\varnothing\}$ and $\{\R^d\}$.  
\end{cor}

Note that in ergodic theory, a classification result for minimal
subsystems is a topological analogue of a classification
of invariant measures. Indeed the measure classification problem for
$(\X, G)$ has been raised by Marklof \cite[\S20]{Marklof_survey},
who has shown (in a series of works with Str\"ombergsson, see
\cite{Marklof_survey, MS} and references therein) that such a classification is of
great interest for a variety of problems in mathematical
physics. Besides the Dirac measures on the fixed points above, there
are additional measures provided by the Poisson point process, and
natural  probability measures on spaces of grids and cut-and-project
sets, embedded in $\X$.

We also provide an independent direct proof of Theorem
\ref{thm:No_Gowers-Danzer_set}. This second proof can be made quantitative,
see Proposition \ref{prop:quantitative}. We deduce the following
result about certain ``$\varepsilon$-nets'' in ``range spaces'' (see
\cite[\S14.4]{AS} and \cite[\S10]{Mato} for definitions and further reading about
these notions):  

\begin{thm}\label{thm:CG}
For every $\varepsilon>0,$ if $N_\varepsilon\subseteq [0,1]^d$
intersects every convex set of volume $\varepsilon$ in $[0,1]^d$, then
there exist a convex set $K\subseteq [0,1]^d$ of volume $\varepsilon$
with $\#(N_\varepsilon\cap K) = \Omega(\log\log(1/\varepsilon))$.
\end{thm}

Additional results and open questions are discussed in \S \ref{sec: questions}. 
For more results related to Danzer's questions and weaker
formulations, we refer the reader to the papers \cite{BW, SW, Bishop,
  SS, PT}.  

\medskip

{\bf Acknowledgements.}
We thank Nati Linial for stimulating remarks and suggestions. We
acknowledge the support of ERC starter grant DLGAPS 279893, and the
support of the Israel Science Foundation 2095/15.

\section{Preliminaries}\label{sec:Preliminaries}
For $x\in\R^d$ and $r>0$ we denote by $B(x,r)$ the ball of radius $r$
centered at $x$ with respect to the Euclidean norm $\norm{\cdot}$, by $B_r\df B(0,r)$, and by $\overline{A}$ the closure of a set $A\subseteq\R^d$. For a set $F\subseteq\R^d$ denote its $\vre$-neighborhood by
$U_\varepsilon(F)\df \bigcup\{B(x,\varepsilon):x\in F\}$. The notation $\#S$ denotes
the cardinality of a set $S$. 

We will need the following theorem (see \cite[Lecture 3]{Ball}, \cite{John}).
\begin{thm}[John's Theorem]\label{thm:John}
For every convex set $K\subseteq\R^d$ there exist ellipsoids
$E_1\subseteq K\subseteq E_2$ such that $\Vol(E_2)/\Vol(E_1)\le C_d$,
where $C_d$ is a constant that depends only on $d$. 
\end{thm}

Let $\X, G$ be as in the introduction, and for $F_1,F_2\in \X$ define
\eq{eq: as in intro}{
D(F_1,F_2)=\inf\left (\left\{
\varepsilon
 : 
\begin{matrix} 
F_1\cap B_{1/\varepsilon}\subseteq U_\varepsilon(F_2) \\ 
F_2\cap B_{1/\varepsilon}\subseteq U_\varepsilon(F_1) 
\end{matrix}\right\}  \cup \{1\} \right).
}

The following facts are standard, see e.g. \cite{ Gelander, delaharpe, Lenz}:

\begin{prop}
\begin{itemize}
\item
$D$ is a complete metric on $\X$.
\item
With this metric, $\X$ is homeomorphic to the space of nonempty compact
subsets of the one-point compactification $\R^d \cup \{\infty\}$,
equipped with the Hausdorff metric, via the map $F \mapsto F \cup
\{\infty\}$. In particular $\X$ is compact.  
\item
If $(F_n)$ is a convergent sequence in $\X$ then 
$$
\lim_{n \to \infty} F_n = \{x \in \R^d: \exists x_n \in F_n \text{ such that } x_n \to_{n
  \to \infty}  x\}.
$$
\end{itemize}
\end{prop}

The topology induced by this metric is also called the
\emph{Chabauty-Fell topology}. It will be the only topology we
consider on $\X$ and thus in the sequel we will omit reference to the
metric $D$.

The standard affine action of $G$ on $\R^d$ induces a $G$-action on
$\X$. The following lemma connects the properties considered by Danzer
and Gowers, to the dynamical system $(\X,G)$.

\begin{lem}\label{Lemma:Danzer_dynamical_characteriztion}
For a set $Y\subseteq\R^d$ we have
\begin{itemize}
\item[(a)]
$Y$ is a Danzer set $\iff$ $\varnothing\notin \overline{G.Y}$. 
\item[(b)]
If there exist $\vre, C>0$ such that for every convex set $K$ of volume $\vre$
we have $\#(K\cap Y)\le C$, then $\R^d\notin \overline{G.Y}$. 
\end{itemize}
\end{lem}

\begin{proof}
To prove (a), note that $G$ acts transitively on the collections of
ellipsoids with the same volume in $\R^d$, so using Theorem
\ref{thm:John} we have the following: 
\[
\begin{split} 
Y \text{ is not Danzer } & \iff \text{for any } T \text{ there is
  an ellipsoid of Vol} \geq T \text{ disjoint from } Y 
\\
&\iff \forall r >0\:\:\exists g\in G \text{ such that } g^{-1}.B_{r}\cap Y=\varnothing \\
&\iff \forall r>0\:\:\exists g\in G \text{ such that } B_{r} \cap g.Y=\varnothing \\
&\iff \varnothing\in \overline{G.Y}.
\end{split}
\]
Statement (b) follows from the fact that for sufficiently small $\delta$
(depending on $\vre$ and $C$), a set consisting of at most $C$ points cannot be
$\delta$-dense in the ball of volume $\vre$. 
\end{proof}

For $s>0$, we say that $Y\subseteq\R^d$ is {\em Danzer with volume
  parameter $s$} if for any 
closed  convex set $K$ of volume $s$ we have $Y \cap K\neq \varnothing.$
\begin{lem}\label{Lemma:Same_s_parameter}
The set of Danzer sets with volume parameter $s$ is closed in $\X$. 
\end{lem}

\begin{proof}
Let $(Y_n)$ be a sequence of Danzer sets with volume parameter $s$ in $\R^d$
such that $Y_n\xrightarrow{n\to\infty} Y$,
and assume by way of contradiction that $Y$ misses a closed convex set $A$ of volume
$s$. 
Since each $Y_n$ is a Danzer set with
parameter $s$, let $p_n\in A\cap Y_n$. Since $A$ is compact, the
sequence $(p_n)$ has a subsequence 
$(p_{n_i})$ that converges to a point $p\in A$. This implies that $p\in
Y=\lim_{i\to\infty}Y_{n_i}$, a contradiction.  
\end{proof}

Motivated by Lemmas \ref{Lemma:Danzer_dynamical_characteriztion} and
\ref{Lemma:Same_s_parameter} we make the following definitions. For a 
subgroup $G_0 \subset G$, and $Y \in \X$, we say that $Y$ is
{\em Danzer for $G_0$} if $\varnothing \notin \overline{G_0.Y}$. Also for $r>0$, we say that $Y$
is {\em Danzer for $G_0$ with parameter $r$} if for 
any $g_0 \in G_0,
Y \cap g_0 \overline{B_{r}} \neq \varnothing.$ The arguments used in
the proofs of Lemmas \ref{Lemma:Danzer_dynamical_characteriztion} and \ref{Lemma:Same_s_parameter} show:
\combarak{To be safe, there is a proof in the latex file which is
  commented out.} 

\begin{prop}\name{prop: summary}
Let $G_0 ,r $ be as above. Then:
\begin{itemize}
\item[(i)]
$Y$ is Danzer for
$G_0$ if and only if there is $r>0$ such that $Y$  is Danzer for $G_0$ with
parameter $r$. 
\item[(ii)]
The collection of $Y \in \X$ which are Danzer for $G_0$ is 
$G_0$-invariant. 
\item[(iii)]
The collection of $Y \in \X$ which are Danzer for $G_0$ with parameter
$r$ is closed and $G_0$-invariant. 
\end{itemize}
\end{prop}


\section{Proofs of the main theorems}\label{sec:main_proofs}
We introduce the following notation. We denote by 
$\xx=(x_1,\ldots,x_d)$ a vector in $\R^d$ and its coordinates, and by
$\ee_1,\ldots,\ee_d\in\R^d$ the standard basis vectors. For $1\le k\le
d$ we write 
$$V_k \df \spa \left(
\ee_i : 2\le
i\le k \right);$$ 
in particular $V_1=\{0\}$. We let $P_k: \R^d \to V_k$ be the
orthogonal projection onto $V_k$. The 
\emph{$x_i$-axis} refers to the set $\spa (\ee_i)$, and given $\yy\in
V_{d}$ by a \emph{horizontal line 
  through $\yy$} we mean the affine line parallel to the $x_1$-axis
through $\yy$, that is, the set
$\{\xx\in\R^d:(x_2,\ldots,x_d)=\yy\}$. 
Finally let $U_0, H_0, H$ be the subgroups of $\SL_d(\R)$
defined in the introduction.

\subsection{A dynamical proof}
For the proof of Theorem \ref{thm:classifying_minimal_subsystems} we first prove the following proposition.
\begin{prop}\label{prop:constructing_line}
For every $S\in\X$, if $S$ is Danzer for $H$ then there
exists $Y\in\overline{U_0.S}$ such that $Y$ contains the $x_1$-axis.  
\end{prop}

\begin{proof}
By Proposition \ref{prop: summary}(i), there is $r>0$ such that $S$
is Danzer for $H$ with parameter $r$, and by Proposition
\ref{prop: summary}(iii), the same is true for any $Y$ in
$\overline{H.S}$. By definition, for any $g \in H$, $S \cap g\overline{B_{r}} \neq \varnothing.$ In
particular, since $U_0 \subset H$,   
every $Y\in\overline{U_0.S}$ intersects every translate of $g_t
\overline{B_{r}}$, where $g_t$ is as in \equ{eq: defn gt}.

It suffices to show that for every $N\in\N$ and $\varepsilon>0$ there
exists a set $Y_{\varepsilon,N}\in\overline{U_0.S}$ such that
$$\{(n\varepsilon,0,\ldots,0):n\in\Z, \absolute{n}\le N\}\subseteq
Y_{\varepsilon,N}.$$
We fix $\varepsilon$ and use induction on $N$. 
For $N=0$, we need to find an element of $\overline{U_0.S} \subset \X$ which
contains the origin. For $t>0$, $g_t\overline{B_{r}}$ is a closed ellipsoid
which is long in the $x_1$ direction and small in all the other
coordinate directions. Therefore, given $\eta>0$ we can choose $t>0$ large
enough so that the image of $g_t \overline{B_{r}}$ under the projection
$P=P_{d}: \R^d \to V_{d}$  is contained in a ball of diameter less than $\eta$. There
is therefore a translate $B'$ of $g_t\overline{B_{r}}$ such that for any $\xx \in B',$
\eq{eq: holds}{
0 < \|P(\xx)\| < \eta.
} 
Since $S$ intersects every translate of $g_t \overline{B_{r}}$, there is a
point $\pp=\pp_{0,\eta}\in S \cap B'$. The group $U_0$ in \equ{eq:
  def U} acts on $\R^d$ as follows:
$$
u(\mathbf{a}). \xx = \left(x_1 + \sum_{i=2}^d a_i x_i, x_2, \ldots, x_d \right).
$$
In particular it shears along
horizontal lines, keeps the $x_1$-axis fixed, satisfies $P(u.\xx) =
P(\xx)$, and if $P(\xx) \neq
0$ then the $x_1$-coordinate of $u(\mathbf{a})\xx$ can be made
arbitrary by using suitable $\mathbf{a}$. Thus we
can find $ u_\eta \in U_0$ such that the $x_1$-coordinate of the point
$u_\eta.\pp$ is $0$, and such that $\|P(u_\eta.\pp)\| < \eta$. 
Since $\eta$ was
arbitrary, and $\X$ is compact, passing to a subsequence and taking a convergent
subsequence we obtain a set $Y_{\varepsilon,0}\in\overline{U_0.S}$ that contains
the origin. 

The induction step is similar. Let
$Y_{\varepsilon,N}\in\overline{U_0.S}$ be the set that is obtained from
the induction hypothesis. For an arbitrary $\eta>0$ choose $t$ so that
the diameter of $P(g_t\overline{B_{r}})$ is less than $\eta$. Let
$B'$ be a translate of $g_t \overline{B_{r}} $ so that \equ{eq: holds}
holds for any element of $B'$.
Then $Y_{\varepsilon, N} \cap
B' \neq \varnothing.$
Let $\pp \in Y_{\varepsilon, N} \cap B' $, then there is some $u_\eta\in U_0$
such that the $x_1$-coordinate of the point
$u_\eta.\pp$ is equal to $(N+1)\varepsilon$, and
$\|P(u_\eta. \pp)\| < \eta$. Letting $\eta \to 0$ and taking
subsequences we find
$Y'_{\varepsilon,N+1}\in\overline{U_0.Y_{\varepsilon, N}}\subseteq
\overline{U_0.S}$ that 
contains the set
$$\{(n\varepsilon,0): n = -N, -(N-1),\ldots, N, N+1\}.$$
Using the set
$Y'_{\varepsilon,N+1}$ instead of $Y_{\varepsilon, N}$ and a similar argument we obtain the
required set $Y_{\varepsilon,N+1}$. This completes the proof.
\end{proof}

\begin{proof}[Proof of Theorem \ref{thm:classifying_minimal_subsystems}]
We set $W_k\df \spa\left (\ee_1,\ldots,\ee_k\right )$, and prove the following
claim by induction on $k$: for every $F\in\X$ which is Danzer for $H$,
there exists a set
$Z_k\in\overline{U_0\ltimes V_k.F}$ that contains $W_k$.  Note that
here we have identified the subspaces $V_k$ with
subgroups of the group of translations $\R^d$.

Proposition \ref{prop:constructing_line} proves the case
$k=1$. Suppose the statement is valid for $k \geq 1$, and we prove its
validity for $k+1$. Let $Z_k^{(0)}=Z_k\in\overline{U_0\ltimes V_k.F}$ be the set
obtained from the induction hypothesis, and let $\varepsilon>0$. Let
$S^{(1)}\df Z_k^{(0)}+\varepsilon \ee_{k+1}$, the set obtained by
translating $Z_k^{(0)}$ by $\varepsilon \ee_{k+1}$. By Proposition
\ref{prop: summary}, $S^{(1)}$ is also Danzer for $H$. 
By the induction
hypothesis there is a set 
$$Z_k^{(1)}\in\overline{U_0\ltimes
  V_k.S^{(1)}}
\subset \overline{U_0 \ltimes V_{k+1} .F}$$
that contains $W_k$. Note that all the subspaces $W_j$
and their translates are $U_0$-invariant, the subspace $W_k$ and its
translates are $V_k$-invariant, and the action of $V_k$ does not
change the $x_j$-coordinates for $j>k$. Therefore $Z_k^{(1)}$ and any
element in its orbit-closure under $U_0 \ltimes V_{k+1}$, 
contains both $W_k$ and its translate by $\varepsilon \ee_{k+1}$. By
repeating the above argument for every $\ell\in\N$ we obtain sets
$Z_k^{(\ell)}\in\overline{U_0\ltimes V_{k+1}.F}$ such that
$Z_k^{(\ell)}$ contains all the $k$-dimensional hyperplanes
$W_k+i\varepsilon \ee_{k+1}$, for $0\le i\le\ell$. Now set
$Z_{k+1,\varepsilon}\in\overline{U_0\ltimes V_{k+1}.F}$ to be a limit
point of the sequence $\left(Z_k^{(2n)}-n\varepsilon
\ee_{k+1} \right)$, then $Z_{k+1,\varepsilon}$ contains a
collection of $k$-dimensional hyperplanes that is $\varepsilon$-dense
in $W_{k+1}$. Taking $\varepsilon\to 0$ we obtain the required set
$Z_{k+1}\in\overline{U\ltimes V_k.F}$. 
\end{proof}

\begin{proof}[Proof I of Theorem \ref{thm:No_Gowers-Danzer_set}]
Since $Y$ is Danzer, by Lemma
\ref{Lemma:Danzer_dynamical_characteriztion}(a) we have $\varnothing
\notin \overline{G.Y}$ and in particular $\varnothing \notin
\overline{U.Y}. $ By Theorem \ref{thm:classifying_minimal_subsystems}, $\R^d \in
\overline{H.Y}$, so 
by Lemma \ref{Lemma:Danzer_dynamical_characteriztion}(b), for all
positive $\vre, C$ there is a convex subset of $\R^d$ of volume $\vre$ containing more
than $C$ points of $Y$, and the statement follows via John's theorem. 
\end{proof}

\subsection{A direct proof} 
Recall our notation that $B_r$ is the Euclidean ball of radius $r$
centered at the origin. A {\em closed centered ellipsoid} is the image of the closed
unit ball $\overline{B_1}$ under a nonsingular linear map $\R^d \to
\R^d$. 
Let $\beta_d$ be
the volume of a ball of radius $1$ in $\R^d$. For $r>0$ and a point
$\xx\notin B_r$ we define a closed centered ellipsoid $E(r,\xx)$
containing $B_r \cup \{\xx\}$, as follows. If $\xx = t\ee_1$ for
$t>r$, then $E(r, \xx)$ is the image of $\overline{B_1}$ under the
linear transformation whose matrix is $\diag  (t/r, 1, \ldots,
1)$. For general $\xx$, let $\Theta$ be an orthogonal linear
transformation with $\Theta (\ee_1) = \xx/\|\xx\|$, and let $E(
r, \xx) = \Theta (E(r, \|\xx\|\ee_1)).$ Clearly 
$$\Vol(E(r, \xx)) =
\beta_d r^{d-1} \|\xx\|.$$ 

\begin{proof}[Proof II of Theorem \ref{thm:No_Gowers-Danzer_set}]
Let $Y\subseteq\R^d$ be a Danzer set with volume parameter $s$, and by
rescaling we may assume that $s$ is the volume of the ball of diameter
$1/2$. Assume with no loss of generality that 
\eq{eq: assume wnlg}{
\vre < \vre_0,
\ \text{ where } \vre_0 = \frac{\beta_d}{2^{d-1}}. 
}

The proof is by induction on $n$. 
The $n=1$ case is true since $Y\neq\varnothing$, so we assume the
validity of the statement for $n$, and prove it for $n+1$. Let $E_n\subseteq\R^d$ be
an ellipsoid with $\#(E_n\cap Y)\ge n$ and  
\begin{equation}\label{eq:No_Gowers-Danzer_set-Vol(K_n)}
\Vol(E_n)<\vre' = \beta_d^{-1/(d-1)}\, \vre^{d/(d-1)} =
\left(\frac{\varepsilon}{\beta_d}\right)^{d/d-1} \beta_d. 
\end{equation}
The group $G$ acts transitively on closed ellipsoids of the
same volume, so let $g\in G$ be such that $g.E_n$ is a ball centered at the origin,
and denote its radius by $r$. Then $\#(g.E_n\cap g.Y) \ge n$ and the choice
of $\vre'$ guarantees that
$$
r^d\beta_d=\Vol(g.E_n)=\Vol(E_n)\stackrel{(\ref{eq:No_Gowers-Danzer_set-Vol(K_n)})}
<\left(\frac{\varepsilon}{\beta_d}\right)^{d/d-1}\beta_d,     
$$
and hence $r^{d-1}\beta_d<\varepsilon$. From \equ{eq: assume wnlg} we
find that $r<1/2$. Let $D \subseteq B_1$ be a ball
of diameter $1/2$ that is disjoint from $g.E_n$. Since $g.Y$ is also a 
Danzer set with parameter $s$, $D$ contains a point $\pp\in g.Y$. Set
$E_{n+1}'\df E(r,\pp)$, and $E_{n+1}\df
g^{-1}.E_{n+1}'$. Since $\norm{\pp}<1$, we have that 
$\Vol(E_{n+1}') \leq r^{d-1} \beta_d$. Thus 
\[\Vol(E_{n+1})=\Vol(E_{n+1}')  \le r^{d-1}\beta_d
< \varepsilon,\]  
and $$\#(E_{n+1}\cap Y) = \#(E'_{N+1} \cap g.Y) \ge n+1.$$
\end{proof}

\section{A finitary version and a quantitative result}\label{sec:Finitary_version}
Consider the following finitary version of
Gowers' question. 
\begin{ques}
Is there a constant $C>0$ such that for every $\varepsilon>0$ there
exists a set $N_\varepsilon\subseteq [0,1]^d$ such that for every
convex set $K\subseteq [0,1]^d$ of volume $\varepsilon$ we have $1\le
\#(N_\varepsilon\cap K)\le C$? 
\end{ques}

In this section we prove 
Theorem \ref{thm:CG}, which  implies a negative answer to this
question. We will need the following proposition.  

\begin{prop}\label{prop:quantitative}
Suppose that $Y\subseteq\R^d$ is a Danzer set with volume  parameter $s$, then
for every $n\in\N$ there exist a convex set $K_n$ such that
$\Vol(K_n)=s$, $0\in K_n$, $\#(Y\cap K_n)\ge n$, and 
\eq{eq: what we want}{
\diam(K_n)\le
C_{d,s}\cdot 4^{\frac{d^{n}}{(d-1)^{n-1}}}, } 
where $C_{d,s}$ is the diameter of
a ball of volume $s$ in $\R^d$.    
\end{prop}

\begin{proof}
We retain the notations as in \S
\ref{sec:main_proofs}, and repeat the idea of Proof II of Theorem
\ref{thm:No_Gowers-Danzer_set} using only elements of the group
$\SL_d(\R)$, and keeping track of the elements that are
used. Since both sides of the inequality \equ{eq: what we want}
scale by the same amount under dilations, for convenience we assume
once more that $s$ is the volume 
of a ball of diameter $1/2$. It follows that there
is some $\yy_1\in Y$ with $\norm{\yy_1}\le 1/2$. 
Set 
\begin{equation}\label{eq:epsilon_1_and_k}
\varepsilon_1^{-1}=4^{\left( \frac{d}{d-1} \right)^{n-1}}
\end{equation}
(the reason for this choice will become clear below). 
%
If $\|\yy_1\| \leq \vre_1$ let $h_1$ be the identity map, and
otherwise let $h_1\in \SL_d(\R)$
be the linear transformation which multiplies 
$\yy_1 
$ by a scalar  to have length $\varepsilon_1$, and uniformly
dilates the perpendicular subspace $\yy_1^\perp$. Note that in both
cases  the operator
norm of $h_1^{-1}$  satisfies $\left\|h_1^{-1} \right \|_{\mathrm{op}} \leq \vre_1^{-1}$.  

 Now let $D_2\subseteq B_1$ be a ball of diameter
$1/2$ which is disjoint from $B_{\varepsilon_1}$. Note that $h_1.Y$ is
also a Danzer set with parameter $s$, hence there is a point $\yy_2\in
D_2\cap h_1.Y$. Let
$E_2=E\left(\varepsilon_1,\frac{\yy_2}{\norm{\yy_2}} \right)$,
then $\yy_2\in E_2$ since $\norm{\yy_2}<1$, and we have
$\Vol(E_2)=\varepsilon_1^{d-1}\beta_d$. Let $h_2\in \SL_d(\R)$ be the
element that maps $E_2$ to a ball, whose radius is  $\varepsilon_2 \df \vre_1^{1-1/d}$. Note that $\left\|h_2^{-1} \right\|_{\mathrm{op}} =
\vre_2^{-1}$, since $h_2^{-1}$ maps a vector of length $\vre_2$ to a
unit vector and uniformly contracts the space ortohogonal to this vector. Repeat this procedure to obtain the following
data for a positive integer $k$:  
\begin{itemize}
\item
a ball $D_k\subseteq B_1$ of diameter $1/2$ which is disjoint from $B_{\varepsilon_{k-1}}$,
\item
a point $\yy_k\in D_k\cap (h_{k-1}\cdots h_1.Y)$, with $\norm{\yy_k}<1$,
\item
an ellipsoid $E_k=E\left(\varepsilon_{k-1},\frac{\yy_k}{\norm{\yy_k}} \right)$
with $\yy_k\in E_k$, and $\Vol(E_k)=\varepsilon_{k-1}^{d-1}\beta_d$, 
\item
an element $h_k\in \SL_d(\R)$ that maps $E_k$ to a ball,
\item
and a number $\varepsilon_k$ which is the radius of that ball, such
that $\left\|h_k^{-1}\right\|_{\mathrm{op}} = \vre_k^{-1}$. 
\end{itemize}
We can repeat the procedure to go from step $k-1$
to step $k$, as long as
$2\varepsilon_{k-1}\le 1/2$. For every $k$, since $h_k$ is volume
preserving, we have
$\varepsilon_{k-1}^{d-1}\beta_d=\Vol(E_k)=\Vol(B_{\varepsilon_k})=\varepsilon_k^d\beta_d$,
and hence  
$$\varepsilon_k=\varepsilon_{k-1}^{1-1/d}.$$
Therefore $(\varepsilon_k)$ is an increasing sequence that approaches $1$, and
\begin{equation}\label{eq:epsilon_1_to_epsilon_k}
\varepsilon_k=\varepsilon_1^{(1-1/d)^{k-1}}.
\end{equation}
Since the operator norm is submultiplicative, we have from
\equ{eq:epsilon_1_to_epsilon_k} that 
\begin{equation}\label{eq:expansion_of_product}
\left\|h_1^{-1}\cdots h_k^{-1} \right\|_{\mathrm{op}} \leq
\prod_{i=1}^k\varepsilon_i^{-1} = 
\vre_1^{-
m_k},
\end{equation}
where 
\begin{equation}\label{eq: 2}
m_k = \sum_{i=1}^k \left(1-\frac{1}{d}\right)^{i-1} = 
d \left( 1- \left( 1-\frac{1}{d} \right)^k
\right).
\end{equation}

Given $n\in\N$, after $n$ steps we obtain a number $\varepsilon_n$ and
matrices $h_1,\ldots,h_n\in \SL_d(\R)$ such  that $h = h_n \cdots h_1$
satisfies
$\#(h^{-1}.\overline{B_{\varepsilon_n}}\cap
Y)=\#(\overline{B_{\varepsilon_n}}\cap h.Y)\ge n$. Set
$K_n=h^{-1}.\overline{B_{\varepsilon_n}}$. The choice of
$\varepsilon_1$ in \equ{eq:epsilon_1_and_k} and
\equ{eq:epsilon_1_to_epsilon_k} ensure that 
$$2\varepsilon_n=1/2, \ \Vol(K_n)=s \ \text{ and } \diam (B_{\vre_n}) = C_{d,s}.$$ 
Plugging into 
(\ref{eq:expansion_of_product}) and (\ref{eq: 2}), we find 
\[\left\| h^{-1} \right\|_{\mathrm{op}} \le 4^{m_n\left(\frac{d}{d-1}
  \right)^{n-1}} \le  
4^{\frac{d^{n}}{(d-1)^{n-1}}}
\]
This implies \equ{eq: what we want}. 
\end{proof}


By using $s=1$ in 
Proposition \ref{prop:quantitative}, and noting that $C_{d,1}
=2\beta_d^{-1/d}$,  we obtain the following: 
%
\begin{cor}\label{cor:quantitative}
Let
$$
\alpha_{d,n}\df 2\beta_d^{-1/d}\cdot 4^{\frac{d^{n}}{(d-1)^{n-1}}},
$$
and let $Q_{d,n}$ be the closed cube of edge length $2\alpha_{d,n}$, centered at the origin.

For every $n\in\N$ and $d\ge 2$, if $Y\subseteq
Q_{d,n}$ intersects every convex set of volume $1$ in $Q_{d,n}$ then
there is a convex set $K\subseteq Q_{d,n}$ of volume $1$ so that $0\in
K$ and $\#(K\cap Y)\ge n$.   
\end{cor}

\begin{proof}[Proof of Theorem \ref{thm:CG}]
Without loss of generality we prove the statement for $[-1/2,1/2]^d$
instead of $[0,1]^d$. Let $\varepsilon>0$ and suppose that
$N\subseteq[-1/2,1/2]^d$ intersects every convex set of volume
$\varepsilon$ in $[-1/2,1/2]^d$. Let $\varphi$ be the linear
transformation $\varphi(\xx ) = \vre^{-1/d} \xx$. 
Let $n\in\N$ such that  
\begin{equation}\label{eq:connects_n_and_epsilon}
\alpha_{d,n}\le \varepsilon^{-1/d}/2<\alpha_{d,n+1}. 
\end{equation}
By the left-hand side of \equ{eq:connects_n_and_epsilon},
$Q_{d,n}\subseteq\varphi([-1/2,1/2]^d)$, and by the right-hand
side, $n = \Omega(\log\log(1/\varepsilon))$. Observe 
that $\varphi$ maps convex sets of volume $\varepsilon$ to convex sets
of volume $1$, and thus $\varphi(N)$ intersects every convex set of volume
$1$ in $\varphi([-1/2,1/2]^d)$, and in particular 
in $Q_{d,n}$. Hence the claim follows from Corollary
 \ref{cor:quantitative}. 
\end{proof}

\section{Questions and concluding remarks}\name{sec: questions}
\subsection{Restricting to subgroups of $G$} 
The proof of Theorem \ref{thm:classifying_minimal_subsystems} and
Corollary \ref{cor:classifying_minimal_subsystems} can be adapted to
prove the following:
\begin{thm}\name{thm: general group}
Let $g_t$ be a one-parameter diagonal subgroup of $\SL_d(\R)$, let
$U_0$ be the expanding horospherical subgroup 
$$
U_0 = \{h \in \SL_d(\R) : g_{t} h g_{-t} \to_{t \to -\infty} 1\},  
$$
let $H_0$ be the subgroup of $\SL_d(\R)$ generated by $\{g_t\}$ and
$U_0$, and let $H = H_0 \ltimes \R^d$. Then for any $Y \in \X$,
$\overline{H.Y}$ contains at least one of the fixed points
$\varnothing, \R^d$.  
\end{thm} 

Recall that $P \subset G$ is called {\em parabolic} if it is connected
and $G/P$ is compact. It is well-known that a parabolic subgroup
contains a group $H$ satisfying the conditions of Theorem \ref{thm:
  general group}. Thus we obtain:

\begin{cor}\name{cor: minimal sets general group}
For any parabolic subgroup $P$ of $G$, the only minimal sets for the
action of $P$ on $\X$ are the fixed points $\{\varnothing\},
\{\R^d\}$. 
\end{cor}

\begin{question}
Which subgroups $H \subset G$ satisfy the conclusion of Corollary \ref{cor: minimal
  sets general group} in place of $P$?  
\end{question}

It is not hard to see that $H = \SL_d(\R)$ does not satisfy the
conclusion of Corollary \ref{cor: minimal sets general group}, and in
fact the collection of minimal sets in this case contains
fixed points, closed orbits which are not fixed points, and minimal
sets which are not closed orbits; for an example of the latter,
consider the case $d=2$, and the orbit-closure of $\{0\} \cup A_1 \cup
A_2 \in \X$, where 
$$A_1 \df 
\left\{ \left (e^{m}, 0 \right) : m \in \Z \right\}, \ \ A_2 \df 
\left\{ \left (e^{n \sqrt{2}}, 0 \right) : n \in \Z \right\}. 
$$
Nevertheless, motivated by \cite{Marklof_survey}, we ask:
\begin{question}
Classify all minimal sets for the action of $\SL_d(\R)$ on $\X$. 
\end{question}
\subsection{Other range spaces}
Several weakenings of the Danzer question have been proposed, in which
the collection (or ``range space'') of convex
subsets of $\R^d$ is replaced with another collection. Recall that an
{\em aligned box}  is a set of the form
$[a_1, b_1] \times \cdots \times [a_d, b_d]$, for some  choice of
$a_i < b_i, \, i=1, \ldots, d$.  We note
that the conclusion of Theorem \ref{thm:No_Gowers-Danzer_set} is not
valid for aligned boxes:  

\begin{prop}
For any $d$ there is $Y \subset \R^d$ and $C>0$ such that for any
aligned box $B$ of volume 1, 
\eq{eq: what we want now}{
1 \leq \# (B \cap Y) \le C.
}
\end{prop}

\begin{proof}
Let $Y \subset \R^d$ be a lattice arising from the geometric embedding of the ring
of integers in a degree $d$ totally real number field. This is a
classical source of examples of lattices with interesting properties,
see \cite{GL}. In particular it is known that $Y$ has a compact orbit
in the space of grids (translated lattices) $G/G(\Z)$, under the group $H_1= A \ltimes \R^d$, where $A$ is the subgroup
of diagonal matrices in $\SL_d(\R)$. The map from the space of grids
to $\X$ is continuous and $G$-equivariant, and therefore $Y$ has a
compact $H_1$-orbit in $\X$ as well. It is also shown in \cite{SS}
that $Y$ is a Danzer set for aligned boxes, i.e. satisfies the
left-hand side of \equ{eq: what we want now}. Suppose by
contradiction that the right-hand side of \equ{eq: what we want now}
fails, that is, for each $n \geq 1$ there is an aligned boxes $B_n$
with $\Vol(B_n)=1$ and $\#(Y \cap B_n) \geq n.$ Since $H_1$ acts
transitively on aligned boxes of volume 1, we can take $h_n \in H_1$ so
that $h_n. B_n = [-1/2, 1/2]^n$. Using the compactness of $H_1.Y$, we can
pass to a subsequence and find that the sequence $\{h_n.Y\}$
converges to a translate of a lattice. At the
same time, $\#(h_n.Y \cap [-1/2, 1/2]^n) \geq n$ for all $n$, so that
the shortest non-zero vector in the lattice $h_n.(Y-Y)$ has length
tending to zero as $n \to \infty$. This is  a contradiction. 
\end{proof}

Note that the group $H_1$ in the above proof admits a cocompact lattice
while the group $G$ does not. This motivates the following:
\begin{question}
For which collections $\mathcal{R}$ of subsets of $\R^d$
of volume 1, is it true that there is $Y \subset
\R^d$ and $C>0$ such that for every $B \in \mathcal{R}$, \equ{eq: what
  we want now} holds? Suppose $\mathcal{R}$ is acted upon transitively
by a subgroup 
$H_{\mathcal{R}} \subset G$, and the collection satisfies the above
property. Does it follow that $H_{\mathcal{R}}$ admits a cocompact lattice? 
\end{question}

\subsection{Improving Theorem \ref{thm:CG} } 
For fixed $d$ and $\vre$, let $f_d(\vre)$ denote the
smallest number $k$ such that for
every $N \subset [0,1]^d$,  such that $N$ intersects every convex
subset of $[0,1]^d$ of volume $\vre$, there is a convex $K \subset
[0,1]^d$ of volume $\vre$ such that $\#(K \cap Y) \geq k$. From
Theorem \ref{thm:CG} we know that $f_d(\vre) = \Omega(\log \log
(1/\vre))$ (where the implicit constant may depend on $d$).  
\begin{question}[Nati Linial]
Is it true that $f_d(\vre) = \Omega(\log 
(1/\vre))$? 
\end{question}

\subsection{Other forms of Danzer's question}
Gowers' question is a weakening of Danzer's in the sense that Gowers asked
about the existence of a set satisfying more stringent conditions than
Danzer does. We record a similar weakened version with a dynamical
flavor. Recall that $Y \subset \R^d$ is 
called {\em uniformly discrete} if 
$$
\inf \{\|y_1 - y_2\| : y_1, y_2 \in Y, y_1 \neq y_2\} >0.
$$
Note that uniform discreteness implies bounded upper density. 

\begin{question}[Michael Boshernitzan]\name{q: Boshernitzan 1}
Is there a uniformly discrete Danzer set $Y \subset \R^d$? 
\end{question}

A simple dynamical argument for the translation action of $\R^d$ on $\X$ implies
that if the answer is positive, then there exists a uniformly discrete
Danzer set which is also {\em repetitive}, i.e. for any $\vre>0$ there
is $R>0$ such that for 
every 
$\yy_1, \yy_2 \in \R^d$ there is $\xx \in B(\yy_2, R)$ such that 
$D(Y - \yy_1, Y - \xx) < \vre$. Here $D$ is the metric
defined in \equ{eq: as in intro}. 




\begin{thebibliography}{99}

\bibitem[AS]{AS} N. Alon and J. H. Spencer, \underline{The
    probabilistic method} (third edition), John Wiley and Sons,
  (2008). 

\bibitem[Ba]{Ball} K. Ball, \underline{An elementary introduction to
    modern convex geometry}, MSRI Publications, (1997). 

\bibitem[Bi]{Bishop} C. Bishop, \emph{A set containing rectifiable
    arcs QC-locally but not QC-globally}, Pure Appl. Math. Q. {\bf 7}
  (2011), no. 1, 121--138. 

\bibitem[BW]{BW} R. P. Bambah and A. C. Woods, \emph{On a problem of
    Danzer}, Pacific J. Math. 37, no. 2 (1971), 295--301. 





\bibitem[CFG]{CFG} H. T. Croft, K. J. Falconer, and R. K. Guy,
  \underline{Unsolved problems in geometry}, Springer (1991).  





\bibitem[F]{Fenchel} W. Fenchel (ed.), Proceedings of the Colloquium on Convexity, 1965,
Kobenhavns Univ. Mat. Inst., (1967), MR 35 $\#$5271.

\bibitem[GS]{Gelander} T. Gelander and A. Solomon, {\em Closed subgroups of locally
    compact groups}, manuscript in preparation.  

\bibitem[Go]{Gowers} T. Gowers, \emph{Rough structures and
    classification}, in \underline{Visions in mathematics},   
N. Alon, J. Bourgain, A. Connes, M. Gromov, V. Milman (eds.), Birkh\"auser (2000), 79--117. 

\bibitem[GL]{GL} P. M. Gruber and  C. G. Lekkerkerker,
  \underline{Geometry of numbers}, Second edition, North-Holland,
  Amsterdam (1987).   
  
\bibitem[H]{delaharpe} P. de la Harpe, {\em Spaces of closed subgroups
    of locally compact groups}, (2008) {\tt
    http://arxiv.org/pdf/0807.2030.pdf} 

 

\bibitem[J]{John} F. John, \emph{Extremum problems with inequalities
    as subsidiary conditions}, in \underline{Studies and essays
    presented to R. Courant on his 60th birthday}, Interscience
  Publishers (1948), 187--204.  

\bibitem[LS]{Lenz} D. Lenz and P. Stollmann, {\em Delone dynamical
    systems and associated random operators}, OAMP Proceedings,
  Constanta 2001; J.-M. Combes, J. Cuntz, G.A. Elliott, G. Nenciu,
  H. Siedentop, S. Stratila (eds.), Theta Foundation.  

\bibitem[Mar]{Marklof_survey} J. Marklof, {\em Kinetic limits of
    dynamical systems}, in: {\underline{ Hyperbolic Dynamics,
      Fluctuations and Large Deviations}},  D. Dolgopyat, Y. Pesin,
  M. Pollicott, and L. Stoyanov (eds.), Proc. Symp. Pure Math.,
  American Mathematical Soc. 2015, pp. 195--223.  

\bibitem[MS]{MS} 
J. Marklof and A. Str\"ombergsson, {\em Free path lengths in
  quasicrystals}, Comm.  Math. Physics {\bf 330} (2014) 723--755. 


\bibitem[Mat]{Mato} J. Matousek, \underline{Lectures on discrete
    geometry}, Springer-Verlag New York, Inc., Secaucus, NJ, USA,
  (2002).  

\bibitem[PT]{PT} J. Pach and G. Tardos, {\em Piercing
    quasi-rectangles --- On a problem of Danzer and Rogers}, J.
  Comb. Th. {\bf 119} (2012) 1391--1397. 

\bibitem[SS]{SS} D. Simmons and Y. Solomon, \emph{A Danzer set for
    axis parallel boxes}, preprint 2014,
  http://arxiv.org/pdf/1409.0926v2.pdf.  


\bibitem[SW]{SW} Y. Solomon and B. Weiss, \emph{Dense forests and
    Danzer sets}, preprint 2014, http://arxiv.org/abs/1406.3807.  


\end{thebibliography}
\end{document}